\documentclass{article}

\usepackage{amsmath}
\usepackage{amssymb}
\usepackage{amsthm}
\usepackage{enumerate}
\usepackage{bbm}

\usepackage{multicol}
\usepackage{amsfonts}
\usepackage{hyperref}

\hypersetup{
    colorlinks=true,
    linkcolor=blue,
    filecolor=magenta,      
    urlcolor=cyan,
    citecolor=blue,
    pdftitle={Overleaf Example},
    pdfpagemode=FullScreen,
    }

\newtheorem{theorem}{Theorem}[section]
\newtheorem{proposition}[theorem]{Proposition}
\newtheorem{corollary}[theorem]{Corollary}
\newtheorem{lemma}[theorem]{Lemma}
\newtheorem{claim}[theorem]{Claim}

\newtheorem{rmk}[theorem]{Remark}
\newtheorem{question}[theorem]{Question}

\newtheorem*{theorem*}{Theorem}
\newtheorem*{corollary*}{Corollary}

\newcommand{\ZZ}{\mathbb{Z}}

\newcommand{\A}{\mathbb{A}}

\newcommand \hyp {\operatorname{Exp}}

\providecommand{\keywords}[1]
{
	\small	
	\textbf{\textit{Keywords:}} #1
}
\providecommand{\subjclass}[1]
{
	\small	
	\textbf{\textit{MSC2020:}} #1
}
\title{HYPERSPACES OF THE DOUBLE ARROW}
\author{ Sebasti\'an Barr\'ia \footnote{Universidad de Concepci\'on, Chile. \\ The author was partially supported by CONICYT-PFCHA Doctorado Nacional 2017-21170738.} } 

\begin{document}
	\maketitle
	
	\begin{abstract}
		Let $\mathbb{A}$ and $\mathbb{S}$ denote the double arrow of Alexandroff and the Sorgenfrey line, respectively. We show that for any $n\geq 1$, the space of all unions of at most $n$ closed intervals of $\mathbb{A}$ is not homogeneous. We also prove that the spaces of non-trivial convergent sequences of $\mathbb{A}$ and $\mathbb{S}$ are homogeneous. This partially solves an open question of A. Arhangel'ski\v{i} \cite{ar87}. In contrast, we show that the space of closed intervals of $\mathbb{S}$ is homogeneous.
	\end{abstract}
	
	\keywords{Double arrow, Hyperspaces, Homogeneous spaces, Sorgenfrey, Nontrivial convergent sequences.}
	
	\subjclass{ 54B20, 54B10, 54F05, 54A20, 54D30, 54E35} 
	
	\section{Introduction.}

Given a space $X$, we denote by $\hyp{(X)}$ the set of all non-empty closed subsets of $X$. For a non-empty open set $V$ of $X$, let $[V]=\{F\in\hyp{(X)}:F\subset V\}$ and $\langle V\rangle=\{F\in\hyp{(X)}:F\cap V\neq\emptyset\}$. The collection of all sets $[V]$ and $\langle V\rangle$ is a subbase for a topology on $\hyp{(X)}$ called the \textit{Vietoris topology}. From now on, $\hyp(X)$ will be considered with this topology. Since $\langle\cup_i V_i\rangle=\cup_i\langle V_i\rangle$ for any collection of non-empty open sets $V_i$; if $X$ is generated by a base, then the Vietoris topology on $\hyp{(X)}$ is generated by the subbase of all sets of the form $[V]$ and $\langle W\rangle$ with $V$ open sets and $W$ basic sets. It is known that if $X$ is compact, then $\hyp (X)$ is also compact. Given a space $X$, a \textit{hyperspace} of $X$ is any subspace of $\hyp{(X)}$. All subsets of $\hyp{(X)}$ will be considered hyperspaces. 

Let $X$ be a Hausdorff space. A set $S\subset X$ will be called a \textit{nontrivial convergent sequence} in $X$ if $S$ is countably infinite and there is $x\in S$ such that $S\setminus V$ is finite for any open neighborhood $V$ of $x$. The point $x$ is called the limit of $S$ and we will say that $S$ converges to $x$. The hyperspace of all nontrivial convergent sequences in $X$ will be denoted $\mathcal{S}_c(X)$. The later space was introduced by García-Ferreira and Ortiz-Castillo in \cite{go} for metric spaces and studied in a more general setting in \cite{mpp}. Today has a great interest among topologists. In the usual sense, a convergent sequence in $X$ is a function $f:\omega\to X$ for which there exists $x\in X$ such that for each open neighborhood $V$ of $x$, there is $m\in\omega$ with $f(n)\in U$ for all $n\geq m$. If $f''(\omega)$ is infinite, then $(\{x\}\cup f''(\omega))\in\mathcal{S}_c(X)$.

A topological  space $X$ is \textit{homogeneous} if for every $x,y\in X$ there exists an autohomeomorphism $h$ of $X$ such that $h(x)=y$. Several classic results on homogeneity involve the study of the $\hyp(X)$. In this paper we are motivated by the following general question.
	
 \begin{question}
		When is $\hyp(X)$ homogeneous?
	\end{question}
	
	In the 1970's, it was shown by R. Schori and J. West \cite{sw} that $\hyp([0,1])$ is homeomorphic to the Hilbert cube. In particular, it is possible that the hyperspace  $\hyp(X)$ is homogeneous while $X$ is not. On the other hand, if  $\kappa>\aleph_1$, then $\hyp{(2^{\kappa})}$ is not homogeneous (see \cite{sce}). Thus, the question of homogeneity of the hyperspace turns out to be quite subtle. 
	
	A. Arhangel'ski\v{i} in \cite{ar87}  asked the following question (which appears in \cite{avm}).
	\begin{question}\label{quesarhan}
		Is the hyperspace  $\hyp{(\A)}$  homogeneous?
	\end{question} 
	
	In this paper we partially answer Question \ref{quesarhan}, by showing that.
	
\begin{theorem} \label{main4}
    $\mathcal{C}_m(\mathbb{A})$ is not homogeneous for any $m\geq 1$.
\end{theorem}
	
	Where $\mathcal{C}_m(\mathbb{A})$ is the hyperspace of $\mathbb{A}$ consisting of all unions of at most $m$ non-empty closed intervals. 
	
	The following result can be seen as a companion of the previous Theorem. 
	
\begin{theorem}\label{main5}
$\mathcal{S}_c(\mathbb{A})$ is homogeneous.
\end{theorem}

In \cite{bm} the authors show that the symmetric products $\mathcal{F}_m(\mathbb{A})$ are not homogeneous for any $m\geq 2$. Since $\mathcal{F}_m(\mathbb{A})\subset\mathcal{C}_m(\mathbb{A})$, now we are a little more closer to answer Question \ref{quesarhan}.

	The paper is organized as follows. In section \ref{seq} we prove that $\mathcal{S}_c(\mathbb{A})$ and $\mathcal{S}_c(\mathbb{S})$ are homogeneous. In section \ref{c1s} we prove that the space of non-empty closed intervals of $\mathbb{S}$ is homogeneous. In section \ref{cma} we give a geometric characterization for spaces of unions of at most $m$ non-empty closed intervals of a compact linearly ordered space and we prove that in the case of the double arrow this spaces are non-homogeneous. Finally, in section \ref{metri} we give a metrization theorem that was obtained in our efforts to prove Theorem \ref{main4} and generalizes a classical result on compact spaces. We will use \cite{en89} as a basic reference on topology and \cite{avm} as a reference for homogeneity and hyperspaces.
	

\section{Homogeneity of $\mathcal{S}_c(\mathbb{A})$ and $\mathcal{S}_c(\mathbb{S})$}\label{seq}

Let $\mathbb{A}_0={]}0,1]\times \{0\}, \mathbb{A}_1=[0,1{[}\times \{1\}$ and $\mathbb{A}=\mathbb{A}_0\cup \mathbb{A}_1$. Define the lexicographical ordering $\langle a,r\rangle \prec \langle b, s\rangle$ if $a<b$ or $a=b $ and $r<s$. The set $\A$ with the order topology is the \textit{double arrow space}.

\begin{proposition}
If $S,T\in\mathcal{S}_c(\mathbb{A})$, then there exists a homeomorphism $h:\mathbb{A}\rightarrow \mathbb{A}$ such that $h''(S)=T$.
\end{proposition}

\begin{proof}
Let $S,T\in\mathcal{S}_c(\mathbb{A})$. First, we will prove that if $S=\{x\}\cup\{x_n:n\in\ZZ^+\}$ and $P=\{\langle 0,1\rangle\}\cup\{\langle 1/2^n,1\rangle:n\in\ZZ^+\}$, then there is a homeomorphism $h_1:\mathbb{A}\rightarrow \mathbb{A}$ such that $h_1''(S)=P$. Since $\mathbb{A}$ is homogeneous, there is a homeomorphism $f:\mathbb{A}\rightarrow \mathbb{A}$ with $f(x)=\langle 0,1\rangle$. We have that the sequence $f(x_n)$ converges to $f(x)=\langle 0,1\rangle$, so we can define inductively $z_1=\max \{f(x_n):n\in\ZZ^+\}$ and $z_m=\max\{f(x_n):n\in\ZZ^+\}\setminus\{z_1,\dots,z_{m-1}\}$ for $ m\geq 2$. By convergence, we can choose a clopen neighborhood $V_1$ of $\langle 0,1\rangle$ such that $f(x_n)\in V_1$ for every $n$ with $f(x_n)\neq z_1$ and $z_1\notin V_1$. Because $\mathbb{A}\setminus V_1$ and $[\langle 1/2,1\rangle,\langle 1,0\rangle]$ are homeomorphic to $\mathbb{A}$ and $\mathbb{A}$ is homogeneous, there exists a homeomorphism $g_1:\mathbb{A}\setminus V_1\rightarrow [\langle 1/2,1\rangle,\langle 1,0\rangle]$ such that $g_1(z_1)=\langle 1/2,1\rangle$. As before, we can choose a clopen neighborhood $V_2$ of $\langle 0,1\rangle$ such that $f(x_n)\in V_2$ for every $n$ with $f(x_n)\neq z_1,z_2$ and $z_1,z_2\notin V_2$. There exists a homeomorphism $g_2:V_1\setminus V_2\rightarrow [(\langle 1/2^2,1\rangle,\langle 1/2,0\rangle]$ with $g_2(z_2)=\langle 1/2^2,1\rangle$. Recursively, we can choose a clopen neighborhood $V_m$ of $\langle 0,1\rangle$ such that $f(x_n)\in V_m$ for every $n$ with $f(x_n)\neq z_1,\dots,z_m$ and $z_1,\dots,z_m\notin V_m$. There exists a homeomorphism $g_m:V_{m-1}\setminus V_{m}\rightarrow [\langle 1/2^{m},1\rangle,\langle 1/2^{m-1},0\rangle]$ with $g_m(z_m)=\langle 1/2^m,1\rangle$.

We define the homeomorphism $g=\bigcup g_m:{]}\langle 0,1\rangle,\langle 1,0\rangle]\rightarrow {]}\langle 0,1\rangle,\langle 1,0\rangle{]}$. Hence, we have the homeomorphism $\overline{g}:\mathbb{A}\rightarrow \mathbb{A}$ with $\overline{g}(x)=g(x)$ if $x\neq\langle 0,1\rangle$ and $\overline{g}(\langle 0,1\rangle)=\langle 0,1\rangle$. In this way, $h_1:=\overline{g}\circ f$ is the desired homeomorphism. 

Finally, by the previous argument there is a homeomorphism $h_2:\mathbb{A}\rightarrow \mathbb{A}$ such that $h_2''(P)=T$. Therefore, the homeomorphism $h:=h_2\circ h_1$ is as required.     
\end{proof}

Since the Sorgenfrey line is homeomorphic to $[0,1[$ with the subspace topology, we will assume that the $\mathbb{S}=[0,1[$. In a very similar way we can prove the following.

\begin{proposition}
If $S,T\in\mathcal{S}_c(\mathbb{S})$, then there exists a homeomorphism $h:\mathbb{S}\rightarrow \mathbb{S}$ such that $h''(S)=T$.
\end{proposition}

\noindent{\bf Proof of Theorem \ref{main5}:}

\begin{proof}
Let $S,T\in \mathcal{S}_c(\mathbb{A})$ and $h$ as in the previous proposition. Let us define $\overline{h}:\mathcal{S}_c(\mathbb{A})\rightarrow \mathcal{S}_c(\mathbb{A})$ such that $\overline{h}(X)=h''(X)$. If $X\in \mathcal{S}_c(\mathbb{A})$, then $h^{-1}(X)\in \mathcal{S}_c(\mathbb{A})$, so $\overline{h}(h^{-1}(X))=X$ and $\overline{h}$ is onto. If $X,Y\in \mathcal{S}_c(\mathbb{A})$ and $\overline{h}(X)=\overline{h}(Y)$, then $h''(X)=h''(Y)$, so $X=Y$ by the injectivity of $h$. Hence, $\overline{h}$ is bijective and $\overline{h}(S)=T$. 

We will prove that $\overline{h}$ is continuous. Let $B$ a basic set of $\mathcal{S}_c(\mathbb{A})$. We have two cases. If $B=\mathcal{S}_c(\mathbb{A})\cap [V]$ with $V$ an open set of $\mathbb{A}$, then $\overline{h}^{-1}(B)=\mathcal{S}_c(\mathbb{A})\cap \overline{h}^{-1}([V])=\mathcal{S}_c(\mathbb{A})\cap [h^{-1}(V)]$. If $B=\mathcal{S}_c(\mathbb{A})\cap \langle V\rangle$ with $V$ a basic set of $\mathbb{A}$, then $\overline{h}^{-1}(B)=\mathcal{S}_c(\mathbb{A})\cap \overline{h}^{-1}(\langle V\rangle)=\mathcal{S}_c(\mathbb{A})\cap \langle h^{-1}(V)\rangle$. Therefore, $\overline{h}$ is continuous.

To end, we will prove that $\overline{h}$ is an open map. Let $B$ a basic set of $\mathcal{S}_c(\mathbb{A})$. If $B=\mathcal{S}_c(\mathbb{A})\cap [V]$ with $V$ an open set of $\mathbb{A}$, then $\overline{h}''(B)=\mathcal{S}_c(\mathbb{A})\cap \overline{h}''([V])=\mathcal{S}_c(\mathbb{A})\cap [h''(V)]$. If $B=\mathcal{S}_c(\mathbb{A})\cap \langle V\rangle$ with $V$ a basic set of $\mathbb{A}$, then $\overline{h}''(B)=\mathcal{S}_c(\mathbb{A})\cap \overline{h}''(\langle V\rangle)=\mathcal{S}_c(\mathbb{A})\cap \langle h''(V)\rangle$.
\end{proof}

Analogously, we can prove that

\begin{proposition}\label{seqsor}
$\mathcal{S}_c(\mathbb{S})$ is homogeneous.
\end{proposition}

\section{Homogeneity of $\mathcal{C}_1(\mathbb{S})$} \label{c1s}

Let $\Delta_2=\{(x,y)\in\mathbb{S}^2: x\leq y\}$. Let $\mathcal{C}_n(\mathbb{S})\subset\hyp(\mathbb{S})$ be the hyperspace of all unions of at most $n$ non-empty closed intervals of $\mathbb{S}$. 
 
\begin{proposition} \label{shom}
The function $\rho:\Delta_2\rightarrow\mathcal{C}_1(\mathbb{S})$ defined by $\rho(a,b)=[a,b]$ is a homeomorphism.
\end{proposition}

\begin{proof}
It is easy to see that $\rho$ is a bijection. For the continuity, we will prove that the preimages under $\rho$ of $[V]$ and $\langle W\rangle$, with $V$ an open set and $W=[c,d[$ a basic open set, are open. 

Let $V$ an open set of $\mathbb{S}$. There exists basic intervals $V_j$ such that $V=\bigcup_j V_j$. Let $(a,b)\in \rho^{-1}([V])=\{(x,y)\in \Delta_2:[x,y]\subset \bigcup_j V_j\}$. We define $B=\bigcup\{V_j:[a,b]\cap V_j\neq\emptyset\}$. We have that $B$ is an interval and open set that contains $[a,b]$. Let $(x,y)\in \Delta_2\cap B^2$. Since $x,y\in B$, we have that $[x,y]\subset B\subset \bigcup_j V_j=V$. Therefore, $(a,b)\in \Delta_2\cap B^2\subset \rho^{-1}([V])$ and $\rho^{-1}([V])$ is open. 

Let $W=[c,d[$ a basic interval of $\mathbb{S}$ and $(a,b)\in \rho^{-1}(\langle W\rangle)$. By definition, $[a,b]\cap W\neq\emptyset$. We have two cases. 

Case 1. If $c\leq b<d$, let us consider $(x,y)\in \Delta_2\cap (\mathbb{S}\times W)$. Thus, $[x,y]\cap W\neq \emptyset$. In this way, $(a,b)\in \Delta_2\cap (\mathbb{S}\times W)\subset \rho^{-1}(\langle W\rangle)$. 

Case 2. If $b\geq d$, necessarily $a< d$. Let $(x,y)\in \Delta_2\cap (]\leftarrow,d{[}\times [d,\rightarrow[)$. By definition, $[x,y]\cap W\neq \emptyset$. Therefore, $(a,b)\in \Delta_2\cap (]\leftarrow,d{[}\times [d,\rightarrow[) \subset \rho^{-1}(\langle W\rangle)$. We conclude that $\rho^{-1}(\langle W\rangle)$ is open.

To show that $\rho^{-1}$ is continuous, we will prove that $\rho$ is an open map. Without loss of generality, let $V=\Delta_2\cap (C\times \mathbb{S})$ an open set of $\Delta_2$, with $C$ a basic interval of $\mathbb{S}$ and $[a,b]\in \rho''(V)$. Thus, $(a,b)\in V$, that is to say $a\leq b$ and $a\in C$. Let $B=[a,\rightarrow[$ and consider $[x,y]\in \langle C\rangle\cap [B]$. Since $[x,y]\cap C\neq\emptyset$ and $[x,y]\subset B$, we have that $x\in C$, so $(x,y)\in V$. In this way, $[a,b]\in \langle C\rangle\cap [B]\subset\rho''(V)$. Therefore, $\rho''(V)$ is an open set.
\end{proof}

\begin{corollary}\label{sorinter}
    $\mathcal{C}_1(\mathbb{S})$ is homogeneous.
\end{corollary}

\begin{proof}
   By the previous proposition, $\mathcal{C}_1(\mathbb{S})$ is homeomorphic to $\Delta_2$. By [\cite{bm}, Theorem 1.4] the results holds.
\end{proof}

\begin{question}
    Is the hyperspace $\mathcal{C}_2(\mathbb{S})$ homogeneous?
\end{question}

\section{Non-homogeneity of $\mathcal{C}_m(\mathbb{A})$} \label{cma} 

The purpose of this section is to  prove Theorem \ref{main4}. It will be convenient to introduce some notation.
	
We will think of an element of the finite power $x\in {}^m{X}$ as a function $x:m\to X$. Given a linearly ordered space $X$, let $\Delta_m(X)=\{x\in {}^m{X} : \forall i\in m-1 (x(i)\leq x(i+1))\}$ and let $\mathcal{F}_m(X)$ be the hyperspace of $X$ consisting of all finite non-empty subsets of cardinality at most $m$. Let $\rho:\Delta_m(X)\to \mathcal{F}_m(X)$ be the map given by $\rho(x)=\{x(0),\dots,x(m-1)\}$ and let $\sim $ denote the equivalence relation on $\Delta_m(X)$ defined by $x\sim y$ if and only if $\rho(x)=\rho(y).$ We consider $\Delta_m(X)/{\sim}$ as a topological space with the quotient topology.
	
The following classical fact gives us a more geometric representation of $\mathcal{F}_m(X)$.

\begin{proposition}[\cite{g54}]\label{qhom}
If $X$ is a linearly ordered space, then the  map $\Tilde{\rho}: \Delta_m(X) /{\sim}\\ \to \mathcal{F}_m(X)$ given by $\Tilde{\rho}([x])=\rho(x)$ is a homeomorphism.
\end{proposition} 

  Let $(X,<)$ be a linearly ordered space. For $m\geq 1$, we denote $\mathcal{C}_m(X)\subset\hyp(X)$ as the hyperspace of all unions of at most $m$ non-empty closed intervals in $X$
  . Let $\varrho:\Delta_{2m}(X)\rightarrow\mathcal{C}_m(X)$ be the map defined by $\varrho(x)=\bigcup_{i\in m}[x(2i),x(2i+1)]$ and let $\approx$ the equivalence relation on $\Delta_{2m}(X)$ defined by $x\approx y$ if and only if $\varrho(x)=\varrho(y)$. Let $p:\Delta_{2m}(X)\to\Delta_{2m}(X)/{\approx}$ be the quotient map. We will sometimes write $[x]$ instead of $p(x)$ to represent the equivalence class. We consider $\Delta_{2m}(X)/{\approx}$ as a topological space with the quotient topology.

\begin{proposition}
If $(X,<)$ is a linearly ordered space, then $\varrho$ is continuous.
\end{proposition}

\begin{proof}
We will prove that the preimages under $\varrho$ of $[V]$ and $\langle W\rangle$, with $V$ an open set and $W$ a basic interval, are open. 

Let $V$ an open set of $X$. There exists basic intervals $V_j$ such that $V=\bigcup_{j\in J} V_j$. Let $x\in \varrho^{-1}([V])=\{y\in \Delta_{2m}(X):\bigcup_{i\in m}[y(2i),y(2i+1)]\subset \bigcup_{j\in J} V_j\}$. For each $i\in m$ we define $W_i=\bigcup\{V_j:[x(2i),x(2i+1)]\cap V_j\neq\emptyset\}$. We have that $W_i$ is an open interval that contains $[x(2i),x(2i+1)]$. Let $y\in \Delta_{2m}(X)\cap \prod_{i\in m}W_i^2$. For all $i$, $y(2i)$ and $y(2i+1)$ are in $W_i$, so $\bigcup_{i\in m}[y(2i),y(2i+1)]\subset \bigcup_{i\in m} W_i\subset \bigcup_{j\in J} V_j=V$ . Therefore, $x\in \Delta_{2m}(X)\cap \prod_{i\in m}W_i^2\subset \varrho^{-1}([V])$ and $\varrho^{-1}([V])$ is open. 

Let $W$ be a basic open interval of $X$ and let $x\in \varrho^{-1}(\langle W\rangle)$ be given. By definition, there exists $j$ such that $[x(2j),x(2j+1)]\cap W\neq\emptyset$. If $W={]}\leftarrow,a{[}$, then we define $B=\prod_{i\in 2m} B_i$ with $B_i=X$ if $i\neq 2j$ and $B_{2j}=W$. If $y\in \Delta_{2m}(X)\cap B$, then $[y(2j),y(2j+1)]\cap W\neq\emptyset$, that is to say $\bigcup_{i\in m}[y(2i),y(2i+1)]\cap W\neq\emptyset$. In this way, $x\in \Delta_{2m}(X)\cap B\subset \varrho^{-1}(\langle W\rangle)$. The proof for $W={]}a,\rightarrow{[}$ is similar. When $W={]}a,b[$ we have two cases.

Case 1. $a<x(2j+1)<b$. Define $B=\prod_{i\in 2m} B_i$ with $B_i=X$ if $i\neq 2j+1$ and $B_{2j+1}=W$. If $y\in \Delta_{2m}(X)\cap B$, then $[y(2j),y(2j+1)]\cap W\neq\emptyset$. We have that $\bigcup_{i\in m}[y(2i),y(2i+1)]\cap W\neq\emptyset$. In this way, $x\in \Delta_{2m}(X)\cap B\subset \varrho^{-1}(\langle W\rangle)$.
 
Case 2. $x(2j+1)\geq b$. Necessarily $x(2j)< b$. Define $B=\prod_{i\in 2m} B_i$ with $B_i=X$ if $i\in 2m\setminus\{2j,2j+1\}$, $B_{2j}={]}\leftarrow,b[$ and $B_{2j+1}={]}a,\rightarrow{[}$. If $y\in \Delta_{2m}(X)\cap B$, then $[y(2j),y(2j+1)]\cap W\neq\emptyset$. Therefore, $x\in \Delta_{2m}(X)\cap B\subset \varrho^{-1}(\langle W\rangle)$.  

We conclude that $\varrho^{-1}(\langle W\rangle)$ is open.
\end{proof}

Analogously to Proposition \ref{qhom}, the following result gives us a more geometric representation of $\mathcal{C}_{m}(X)$.

\begin{corollary}\label{ordchar}
If $(X,<)$ a compact linearly ordered space, then the map $\Tilde{\varrho}:\Delta_{2m}(X)/\approx\ \to\mathcal{C}_m(X)$ given by $\Tilde{\varrho}([x])=\varrho(x)$ is a homeomorphism.
\end{corollary}

\begin{proof}
Since $\varrho$ is continuous, we have that $\Tilde{\varrho}$ is a continuous bijection. Let $x\in {}^{2m}{X}\setminus \Delta_{2m}(X)$. There are $i,j\in 2m$ such that $i<j$ and $x(i)>x(j)$. Since $X$ is Hausdorff, there exists two disjoint basic intervals $V$ and $W$ with $W<V$ such that $x(i)\in V$ and $x(j)\in W$. Let $A=\prod_{k\in 2m}X_k$ an open neighborhood of $x$ with $X_k=X$ for all $k\in 2m\setminus\{i,j\}$, $X_i=V$ and $X_j=W$. We have that $x\in A\subset {}^{2m}{X}\setminus \Delta_{2m}(X)$, so $\Delta_{2m}(X)$ is closed in ${}^{2m}{X}$. Since ${}^{2m}{X}$ is compact, so is $\Delta_{2m}(X)$. Therefore, $\Delta_{2m}(X)/\approx$ is compact and $\Tilde{\varrho}$ is a homeomorphism. 
\end{proof}

\begin{rmk}\label{intervals}
We note that $\Delta_2(X)=\Delta_2(X)/\approx$. By the previous Corollary and Proposition \ref{qhom}, we have that $\mathcal{F}_2(X)\cong \mathcal{C}_1(X)$.
\end{rmk}

Let $\pi: \mathbb{A}\to [0,1]$ be the projection onto the first factor $\pi(\langle x,r\rangle )=x$. For any $a\in \mathbb{A}$ we will denote by $\overline{a}$ the constant sequence $a$ of finite length $m$, where the value of $m$ should be understood by context. Let $\pi_i:{}^m{\mathbb{A}}\to \mathbb{A}$ be the projection onto the $i$-th coordinate, and  for any function $h:{}^m{\mathbb{A}}\to {}^m{\mathbb{A}}$, let $h_i=\pi_i\circ h$ denote its $i$-th coordinate function. Recall that a partial function $f:\mathbb{A}\to \mathbb{A}$ is monotone if it is either non-decreasing or non-increasing, and  $f$ is strictly monotone if it is either strictly increasing or strictly decreasing.
	
We recall the following results. 
 
	\begin{proposition}[\cite{bm} Proposition 2.2]\label{basiclema}
		Let $h:\mathbb{A}\to\mathbb{A}$ be a monotone continuous function. Then there is a clopen interval $J$ so that either $h\restriction J$ is constant or $h\restriction J$ is strictly monotone.
	\end{proposition}

    \begin{proposition}[\cite{bm} Proposition 3.2]\label{clopensubsets}
    Every clopen subset of ${}^m{\mathbb{A}}$ is homeomorphic to ${}^m{\mathbb{A}}$.
    \end{proposition}	
    
	\begin{lemma}
		If $\mathcal{C}_m(\A)$ is homogeneous, then it is homeomorphic to ${}^{2m}{\mathbb{A}}$.	
	\end{lemma}
 
	\begin{proof}
		Suppose $\mathcal{C}_m(\A)$ is homogeneous, then there is an autohomeomorphism $h:\Delta_{2m}/{\approx} \to \Delta_{2m}/{\approx}$ such that $h([\overline{\langle 0,1\rangle}])=[x],$ where $x$ is some fixed point such that $\pi(x(0))<\pi(x(1))<\dots<\pi(x(2m-1))$. On one hand, notice that, if  $J_0<\dots<J_{2m-1}$ is a sequence of pairwise disjoint clopen intervals with $x(i)\in J_i$ for $i\in 2m$ and $\max(\pi(J_i))<\min(\pi(J_{i+1}))$ for $i\in 2m-1$, then $p\restriction \prod\limits_{i\in 2m} J_i:\prod\limits_{i\in 2m}J_i\to \Delta_{2m}/{\approx}$ is an embedding. On the other hand, observe that for any $0<\epsilon<1$ the clopen cube ${}^{2m}{[\langle 0,1\rangle,\langle\epsilon,0\rangle]}$ is a saturated neighborhood of $\overline{\langle 0,1\rangle} $ such that $p''({}^{2m}{[\langle 0,1\rangle,\langle\epsilon,0\rangle]})$ is homeomorphic to $\Delta_{2m}/{\approx}.$ Since $h$ is continuous, there is an $\epsilon>0$ such that $h''({}^{2m}{[\langle 0,1\rangle,\langle\epsilon,0\rangle]}/\approx )\subseteq \prod\limits_{i\in {2m}} J_i.$ Thus, we have  that $${}^{2m}{\mathbb{A}}\cong \prod\limits_{i\in {2m}} J_i\cong h''({}^{2m}{[\langle 0,1\rangle,\langle\epsilon,0\rangle]}/{\approx} )\cong {}^{2m}{[\langle 0,1\rangle,\langle\epsilon,0\rangle]}/{\approx}\cong \Delta_{2m}/{\approx}$$ 
	where the second homeomorphism follows from Proposition \ref{clopensubsets}.
	\end{proof}

    We are now ready to prove the main result of the section.	
    \bigskip
	
\noindent{\bf Proof of Theorem \ref{main4}:}

\begin{proof}
		We proceed by contradiction. Suppose that there is a homeomorphism $h: \Delta_{2m}/{\approx} \to {}^{2m}{\mathbb{A}},$ and let $\Gamma=\{[\overline{x}]\in \Delta_{2m}/{\approx} : x\in\mathbb{A}\}.  $ Recall that the diagonal $\{(x,x)\in {}^2{\mathbb{A}}: x\in \mathbb{A}\}$ is not a $G_\delta$ subspace as $\mathbb{A}$ is a non-metrizable compact space.  It follows from this that $\Gamma$ is not a $G_\delta$ in  $\Delta_{2m}/{\approx}$ as otherwise this would imply that $\pi_{\{0,1\}}''(p^{-1}(\Gamma))=\pi_{\{0,1\}}''(\{\overline{x}\in\Delta_{2m}:x\in\mathbb{A}\})=\{(x,x)\in {}^2{\mathbb{A}}: x\in \mathbb{A}\}$ would also be one, where $\pi_{\{0,1\}}:{}^{2m}{\mathbb{A}}\to{}^2{\mathbb{A}} $ denotes the projection onto the first $2$ coordinates. Notice that, since $\mathbb{A}\times {}^{2m-1}\{\langle 0,1\rangle\}=\bigcap_{n\in\omega}\mathbb{A}\times {}^{2m-1}\{[\langle 0,1\rangle,\langle \frac{1}{n},0\rangle]\}\subseteq {}^{2m}{\mathbb{A}}$ and $\mathbb{A}$ is a perfect space, then every closed subset of $\mathbb{A}\times {}^{2m-1}\{\langle 0,1\rangle\}$ is a $G_\delta $ set in ${}^{2m}{\mathbb{A}}.$ Analogously, every closed subset of any line parallel to one of coordinates axis, is also  a $G_\delta $ set in ${}^{2m}{\mathbb{A}}$. We now  consider the embedding $\alpha: \mathbb{A}\to {}^{2m}{\mathbb{A}}$ given by $\alpha(x)=h([\overline{x}])$. By applying Proposition \ref{basiclema} $2m$-times, we can find a clopen interval $J$ such that $\alpha_{j}:=\pi_{j}\circ \alpha\restriction J$ is monotone for every $j\in 2m.$ Since $h''(\Gamma)$ is not a $G_\delta$ in $^{2m}\mathbb{A},$ it follows, by our previous observations, that there exists $j_0\ne j_1\in 2m$  such that $\alpha_{j_0}$ and $\alpha_{j_1}$ are strictly monotone restricted to $J$. We will assume that both $\alpha_{j_0}\restriction J, \alpha_{j_1}\restriction J$ are strictly increasing, as the other cases are analogous. 

  The proof of the following result is is analogous to the proof of [\cite{bm}, Claim 3.5].
  
		\begin{claim}
			There is a countable subset $C\subseteq \pi''(J)$ such that $$\pi(\alpha_{j_0}(\langle a,0\rangle))=\pi(\alpha_{j_0}(\langle a,1\rangle))$$ and $$\pi(\alpha_{j_1}(\langle a,0\rangle))=\pi(\alpha_{j_1}(\langle a,1\rangle))$$ for any $a\in \pi''(J)\setminus C.$ In other words,  $\alpha_{j_k}(\langle a,1\rangle)$  is the immediate successor of  $\alpha_{j_k}(\langle a,0\rangle)$ for $k\in 2.$
		\end{claim}

		For each $a \in A:=\pi''(J)\setminus C,$ let $P_a^-=\alpha(\langle a,0\rangle), Q_a^+=\alpha(\langle a,1\rangle)$ and let 
		$$P_a^+=\alpha(\langle a,0\rangle)\restriction_{(2m\setminus\{j_0\})}\cup\ (j_0,\langle \pi(\alpha_{j_0}(\langle a,0\rangle )),1\rangle)$$
		and 
		$$Q_a^-=\alpha(\langle a,1\rangle)\restriction_{(2m\setminus\{j_1\})}\cup\ (j_1,\langle \pi(\alpha_{j_1}(\langle a,0\rangle )),0\rangle).$$
		
		Pick an element $[x_a]$ belonging to $h^{-1}(\{P^+_a,Q^-_a\})\setminus \Tilde{\varrho}^{-1}([\langle a,0\rangle,\langle a,1\rangle]). $ 
		Observe that, by our choice of $x_a,$  there is a $\ell_a\in 2m $ so that $\pi(x_a(\ell_a))\ne a.$  Let $$A^{P,<}=\{a\in A: h([x_a])=P^+_a, \pi(x_a(\ell_a))< a \},$$ $$ A^{P,>}=\{a\in A: h([x_a])=P^+_a, \pi(x_a(\ell_a))> a \},$$ $$ A^{Q,<}=\{a\in A: h([x_a])=Q^-_a, \pi(x_a(\ell_a))< a\}$$ and $$A^{Q,>}=\{a\in A: h([x_a])=Q^-_a, \pi(x_a(\ell_a))>a\}.$$  We may assume, without loss of generality, that $A^{P,<}$ is uncountable as the other cases are similar.  By successively refining $A^{P,<}$, we can find an uncountable subset $B\subseteq A^{P,<},$ a natural number $\ell$  and a rational number $ r\in \mathbb{Q}$ such that $\ell_a=\ell$ and $\pi(x_a(\ell))<r<a$ for any $a\in B.$
		
		Consider the clopen sets $$U:=\bigcup\limits_{j\in 2m}\pi_j^{-1}( [\langle 0,1\rangle, \langle r,0\rangle ]) \ {\rm and}\ V:=\bigcap\limits_{j\in 2m} \pi_j^{-1}([\langle r,1\rangle, \langle 1,0\rangle ]).$$ 
  
  \begin{claim}
   The sets $U$ and $V$ are saturated.   
  \end{claim}
  
    \begin{proof}
        Let $x\in p^{-1}(p''(U))$. There is $y\in U$ such that $\bigcup_{i\in m}[x(2i),x(2i+1)]=\bigcup_{i\in m}[y(2i),y(2i+1)]$. Since there is $j\in 2m$ and $k\in m$ with $\langle 0,1\rangle\leq y(j)\leq\langle r,0\rangle$ and $y(j)\in [x(2k),x(2k+1)]$, then $\langle 0,1\rangle\leq x(2k)\leq\langle r,0\rangle$. Thus, $x\in U$.

        Let $x\in p^{-1}(p''(V))$. There is $y\in V$ such that $\bigcup_{i\in m}[x(2i),x(2i+1)]=\bigcup_{i\in m}[y(2i),$ $y(2i+1)]$. Since $y(j)\in [\langle r,1\rangle, \langle 1,0\rangle ]$ for any $j\in 2m$, we have that $\bigcup_{i\in m}[x(2i),x(2i+1)]\subset [\langle r,1\rangle, \langle 1,0\rangle ]$ for any $j\in 2m$. It follows that $x(j)\in [\langle r,1\rangle, \langle 1,0\rangle ]$ for any $j\in 2m$, that is to say, $x\in V$. 
    \end{proof}
  
  We have that $\Tilde{U}:=p''(U)$ and $\Tilde{V}:=p''(V)$ form a clopen partition of $\Delta_{2m}/{\approx}.$  Notice that $X:=\{[x_a]: a\in B\}\subset \Tilde{U} $ and $Y:=\{[\overline{\langle a, 0\rangle}]: a\in B\}\subset \Tilde{V}$. Since $B$ is infinite (uncountable) and $\Delta_{2m}/{\approx}$ is compact, then the set of accumulation points $X'$ and $Y'$ of $X$ and $Y$, respectively, are both non-empty. It follows that $X'\cap Y'=\emptyset$. 
  
  \begin{claim}
      The sets $h''(X)=\{P_a^+: a\in B\}$ and $h''(Y):=\{P_a^-: a\in B \}$  have the same accumulations points. 
  \end{claim}
  
  \begin{proof}
     We shall prove that the accumulation points of $h''(X)$ are contained in the accumulation points of $h''(Y)$ as the other case is analogous.  Let $P$ be an accumulation point of $h''(X)$ and let $W:=\prod_{j\in 2m} J_j$ be a clopen neighborhood of $P$ where each $J_j$ is a clopen interval.
     Since $P$ is an accumulation point, then there is an infinite subset $B'\subseteq B$ such that $\{P_a^+: a\in B' \}\subseteq W.$ By construction $P_a^-(j)=P_a^+(j)$ for any $j\in 2m\setminus\{j_0\}$ and $a\in B.$ In particular, $P_a^-(j)\in J_j$ for any $j\in 2m\setminus\{j_0\}$ and $a\in B'$. Observe that $P_a^+(j_0)\ne P_{b}^+(j_0)$ for any $a\ne b \in B$ as $\alpha_{j_0}\restriction J$ is strictly monotone. Thus, there is an infinite subset $B''\subseteq B'$ such that $\pi(P_a^+(j_0))\notin \{\pi(\min(J_{j_0})),\pi(\max(J_{j_0}))\}$ for all $a\in B''$. It follows that, $\{P_a^-: a\in B''\}\subseteq W$ and hence, $P$ is an accumulation point of $h''(Y)$ as required. 
  \end{proof}
  
  Since $h$ is a homeomorphism, then $X $ and $Y$ have the same accumulation points which is a contradiction. This finishes the proof of the Theorem.  
 \end{proof}

It would be interesting to see if the above theorem can be extended to the hyperspace of all finite unions of non-empty closed intervals $\mathcal{C}(\mathbb{A})$.

	\begin{question}
		Is the hyperspace $\mathcal{C}(\mathbb{A})$ homogeneous?
	\end{question}

\section{A metrization theorem} \label{metri}

\begin{proposition}\label{met}
    Let $X$ be a compact Hausdorff space. If there exists a $G_{\delta}$-set $C\subset X^2$ homeomorphic to $X$ such that for every $x\in X$ there are $a\in C$ and a unique $b\in C$ with $x=\pi_1(a)=\pi_2(b)$, then $X$ is metrizable.
\end{proposition}

\begin{proof}
For each $n\in\omega$, let $G_n$ be an open subset of $X^2$ such that $C=\bigcap_{n\in\omega}G_n$ and $G_{n+1}\subset G_n$. Since $X^2$ is normal and $C$ is closed, we can define a sequence of open sets $U_n$ as follows. Let $U_0=G_0$ and for $n>0$ let $U_n$ such that $C\subset U_{n}\subset\overline{U_{n}}\subset U_{n-1}\cap G_{n}$. It follows that $C=\bigcap_{n\in\omega}\overline{U_n}$. Let $x\in X$ and $(s,x)\in C$. For $n\in\omega$, let $U_n[s]:=\{y\in X:(s,y)\in U_n\}$. Hence, $\overline{U_n[s]}\subset U_{n-1}[s]$ for any $n>0$, since \begin{equation*}
    \begin{split}
      \overline{U_n[s]}&=\overline{\pi_2''(U_n\cap(\{s\}\times X))}=\pi_2''(\overline{U_n\cap(\{s\}\times X)})\subset\pi_2''(\overline{U_n}\cap\overline{\{s\}\times X})\\ &\subset
      \pi_2''(U_{n-1}\cap(\{s\}\times X))=U_{n-1}[s]  
    \end{split}
\end{equation*}  where the second equality follows from [\cite{en89}, Corollary 3.1.11].

\begin{claim}
For any open neighborhood $V$ of $x$, there exists $n\in\omega$ such that $$x\in U_n[s]\subset V$$
\end{claim}

\begin{proof}
    We proceed by contradiction. Let $V$ be an open neighborhood of $x$ with $U_n[s]\not\subset V$ for any $n\in \omega$. We choose $x_n\in U_n[s]\setminus V$. It follows that $\bigcap_{n\in\omega}\overline{U_n[s]}=\{x\}$, since if $z\in \bigcap_{n\in\omega}\overline{U_n[s]}\subset\bigcap_{n\in\omega}U_n[s]$, then $(s,z)\in\bigcap_{n\in\omega}U_n\subset C$. Since $X$ is compact, the pseudocharacter and the character of $x$ are equal. Since the sets $U_n[s]$ are open, the pseudocharacter of $x$ is countable. Hence, $X$ is first countable. By [\cite{en89}, Theorem 3.10.31] $X$ is sequentially compact. In this way, there exists a convergent subsequence of $(x_n)(n\in\omega)$, let us say with limit $L$. Since each set $\overline{U_m[s]}$ is closed and contains infinitely many elements of such subsequence, we have that $L$ belongs to each one of them. Thus, $L\in\bigcap_{m\in\omega}\overline{U_m[s]}=\{x\}$. By convergence, there are infinite elements of the subsequence in $V$, which is a contradiction.
\end{proof}

For each $(s,t)\in C$ and $n\in\omega$ we can choose open neighborhoods $B_{s,n}$ and $B_{t,n}$ such that $B_{s,n}\times B_{t,n}\subset U_n$. Since $X$ is compact, there exists a finite subcover $\mathcal{B}_n$ of $\{B_{t,n}:t\in X\}$. We claim that $\mathcal{B}=\bigcup_{n\in\omega}\mathcal{B}_n$ is a countable base for $X$. Let $x\in X$ and $V$ an open neighborhood of $x$. By the claim, there is $n\in\omega$ such that $x\in U_n[s]\subset V$. We choose $B_{z,n}$ from $\mathcal{B}_n$ with $x\in B_{z,n}$. Therefore, $x\in B_{z,n}\subset U_n[s]\subset V$.

By the Urysohn's metrization theorem, $X$ is metrizable.
\end{proof}

As a consequence, we obtain the following classical fact.

\begin{corollary}[\cite{sn45}]
Let $X$ a compact Hausdorff space. If the diagonal of $X$ is a $G_{\delta}$-set, then $X$ is metrizable.
\end{corollary}

\section*{Acknowledgements}
	
The author wants to thanks Professor Carlos Martínez-Ranero for suggesting the problem and the discussions that arose around it.
	
		
		
		
		
		

	\bigskip
	\noindent Sebastián Barr\'ia\\
	Email: sebabarria@udec.cl\\
	
	
	\noindent  
	Universidad de Concepci\'on, Concepci\'on, Chile\\
	Facultad de Ciencias F\'isicas y Matem\'aticas\\
	Departamento de Matem\'atica\\

\end{document}